\documentclass{article}
\usepackage[a4paper]{geometry}
\usepackage{amsmath}
\usepackage{amssymb}
\usepackage{amsthm}
\usepackage{mathrsfs}
\usepackage{mathtools}

\usepackage{tikz}
\usepackage{tikz-3dplot}
\usepackage{cleveref}
\usepackage{tikz-cd}
\usepackage{mathtools}

\usepackage{placeins}

\usepackage{float}

\usepackage{comment}

\usepackage{enumerate}

\usetikzlibrary{calc}

\theoremstyle{definition}
\newtheorem{thm}{Theorem}[section]
\crefname{thm}{Theorem}{Theorems}

\crefname{prop}{Proposition}{Propositions}

\crefname{lem}{Lemma}{Lemmas}

\newtheorem{conj}[thm]{Conjecture}

\crefname{defn}{Definition}{Definitions}

\newtheorem{case}{Case}
\newtheorem*{ack*}{Acknowledgements}

\newcommand{\co}{\operatorname{co}}

\usepackage{titling}

\title{The sharp threshold for Hausdorff convexification under Minkowski addition}
\author{Peter van Hintum}
\allowdisplaybreaks

\begin{document}\maketitle
\begin{abstract}
The Dyn-Farkhi conjecture asserts that the square of the Hausdorff distance from a compact set to its convex hull is subadditive with respect to Minkowski addition  \cite{Dyn01012005}. The conjecture is elementary in dimension 1, was recently proved by Meyer in dimension 2 \cite{meyer2025dyn}, and was disproved in dimensions $n\geq3$ by Fradelizi, Madiman, Marsiglietti, and Zvavitch \cite{fradelizi2018convexification}. The symmetric case $A=B$, however, remained open \cite{fradelizi2018convexification,meyer2025dyn}. We show that the conjecture already fails in this restricted setting. More precisely, for every $n\geq3$, we construct a compact set $A\subset\mathbb{R}^n$
 such that
$$d(A(k))=d(A)>0$$ for every $1\leq k\leq n-1$, where $A(k):=\frac1k (A+\dots+A)$ is the $k$-fold iterated Minkowski average of $A$.
We also prove that the threshold $k=n$ is sharp: for every nonempty compact $A\subset\mathbb{R}^n$ with $n\geq 2$, we have
$$d(A(n))\leq \left(1-\frac{n-1}{n(2n-1)}\right)d(A).$$
\end{abstract}

\section{Introduction}
Given $k,n\in\mathbb{N}$, and sets $A,B\subset\mathbb{R}^n$, let $A+B:=\{a+b:a\in A,b\in B\}$ be the \emph{Minkowski sum}, let $A(k):=\frac1k (\overbrace{A+\dots+A}^{k\text{ times}})$ be the \emph{iterated Minkowski average}, and let $\co(A)$ be the \emph{convex hull} of $A$. Intuitively, $A(k)$ increasingly resembles $\co(A)$ as $k$ grows. Quantitative results capturing this intuition have long received much attention, see e.g. \cite{starr1969quasi,emerson1969asymptotic,schneider1975measure,cassels1975measures,Dyn01012005,fradelizi2016minkowski,fradelizi2018convexification,figalli2015quantitative,van2021sharp,van2023sets}.

One way to quantify this tendency is through $d(A)$ defined as the Hausdorff distance between $A$ and $\co(A)$:
$$d(A):=\sup_{x\in\co(A)}\inf_{y\in A} \|x-y\|_2.$$
The Dyn-Farkhi conjecture \cite{Dyn01012005} asserted the subadditivity of the square of $d(\cdot)$ under Minkowski addition.
\begin{conj}[Dyn-Farkhi Conjecture]
Given $A,B\subset\mathbb{R}^n$ nonempty compact sets, we have
$$d(A+B)^2\leq d(A)^2+d(B)^2.$$
\end{conj}
For $n=1$, the conjecture is easy to verify. Meyer \cite{meyer2025dyn} recently proved the conjecture for $n=2$.
Fradelizi, Madiman, Marsiglietti, and Zvavitch \cite{fradelizi2018convexification} disproved this conjecture if $n\geq 3$ by constructing appropriate sets $A$ and $B$. The sets $A$ and $B$ they constructed were distinct and they asked whether the conjecture might be true if one additionally requires $A=B$. Note in the study of the convexification effect of Minkowski addition, it is often the case that stronger results are true for $A+A$ than for $A+B$, see e.g. \cite{van2021sharp,figalli2023sharp,figalli2025sharp}.

In the symmetric case $A=B$, the conjecture  becomes the following.
\begin{conj}[Dyn-Farkhi Conjecture for $A=B$]
Given a nonempty compact set $A\subset\mathbb{R}^n$, we have
$$d\left(A(2)\right)\leq \frac{d(A)}{\sqrt{2}}.$$
\end{conj}
The purpose of this note is to give a simple counterexample to this weaker conjecture (and thus also to the original Dyn-Farkhi conjecture) in the case $n\geq 3$. In fact, we will show that in this example no convexification occurs until after at least $n$ averages, i.e. $d(A(k))=d(A)$ for all $k\leq n-1$.
\begin{thm}\label{exampleprop}
In any dimension $n\geq 3$, there exists a compact set $A\subset\mathbb{R}^n$ so that for all $1\leq k\leq n-1$, we have
$$d\left(A(k)\right)=d(A)>0.$$
\end{thm}
This theorem almost complements a result by Fradelizi, Madiman, Marsiglietti, and Zvavitch \cite{fradelizi2018convexification} (originally proved in  \cite[Theorem 4]{fradelizi2016minkowski} using a result from \cite{schneider1975measure}) which showed that for $k\geq n$, we have $d(A(k+1))\leq \frac{k}{k+1} d(A(k))$. These two results leave open whether there exist sets with $d(A(n))=d(A)>0$. The second result of this note is to answer this question in the negative.

\begin{thm}\label{ChatProp}
Let $A\subset\mathbb{R}^n$ with $n\geq 2$ be nonempty and compact, then
$$d(A(n))\leq \left(1-\frac{n-1}{n(2n-1)}\right)d(A).$$
\end{thm}

In light of the result from \cite{fradelizi2018convexification}, it would be interesting to know whether \Cref{ChatProp} remains true with $1-\frac{n-1}{n(2n-1)}$ replaced by $1-\frac1n$.

The counterexample we consider is an $(n-1)$-dimensional simplex with $n$ points added near the vertices, see \Cref{n=3pic} for an illustration for $n=3$. It is worth noting that this is also the conjectured extremal configuration for the convexification effect on the volume deficit in \cite{van2020sharpL1}. Moreover, this example works (after a suitable rotation and choosing the points near the vertices sufficiently close) for generalizations of the Hausdorff distance, where the Euclidean norm is replaced by any other norm on $\mathbb{R}^n$ as considered in this context by \cite{fradelizi2018convexification} and more recently in the form of the $\ell_p$-norm by \cite{meyer2026sharp}.

\begin{figure}[t]

    \centering

\tdplotsetmaincoords{70}{115}

\begin{tikzpicture}[scale=3.8]

\tikzset{
  subtle point/.style={
    circle,
    fill=red!55!black,
    draw=white,
    line width=0.35pt,
    inner sep=1.05pt,
    fill opacity=0.72
  },
  base label/.style={
    font=\small,
    fill=white,
    fill opacity=0.82,
    text opacity=1,
    inner sep=1.2pt
  },
  eps label/.style={
    font=\small,
    fill=white,
    fill opacity=0.92,
    text opacity=1,
    inner sep=1.2pt
  },
  eps small/.style={
    red!55!black,
    thick,
    densely dotted
  }
}


\begin{scope}[xshift=-1.15cm, tdplot_main_coords]

  \fill[blue!18, fill opacity=0.85]
      (0,1,0) -- (-0.866,-0.5,0) -- (0.866,-0.5,0) -- cycle;
  \draw[blue!65!black, very thick]
      (0,1,0) -- (-0.866,-0.5,0) -- (0.866,-0.5,0) -- cycle;

  \draw[gray!40, dashed, very thin] (0,1,0) -- (0,1,0.32);
  \draw[gray!40, dashed, very thin] (-0.866,-0.5,0) -- (-0.866,-0.5,0.32);
  \draw[gray!40, dashed, very thin] (0.866,-0.5,0) -- (0.866,-0.5,0.32);

  \node[base label, above=2pt]       at (0,1,0)          {$v_1$};
  \node[base label, below left=1pt]  at (-0.866,-0.5,0) {$v_2$};
  \node[base label, below right=1pt] at (0.866,-0.5,0)  {$v_3$};

  \node[subtle point] at (0,1,0.32) {};
  \node[subtle point] at (-0.866,-0.5,0.32) {};
  \node[subtle point] at (0.866,-0.5,0.32) {};

  \draw[gray!45, thin] (0.866,-0.5,0)    -- (1.13,0,0);
  \draw[gray!45, thin] (0.866,-0.5,0.32) -- (1.13,0,0.32);
  \draw[eps small, <->]
      (1.13,0,0) -- (1.13,0,0.32)
      node[midway, right=2pt, eps label] {$\varepsilon$};

  \node[font=\large] at (0,0,-0.57) {$A$};

\end{scope}

\begin{scope}[xshift=1.15cm, tdplot_main_coords]

  \fill[blue!18, fill opacity=0.85]
      (0,1,0) -- (-0.866,-0.5,0) -- (0.866,-0.5,0) -- cycle;
  \draw[blue!65!black, very thick]
      (0,1,0) -- (-0.866,-0.5,0) -- (0.866,-0.5,0) -- cycle;

  \fill[orange!24, fill opacity=0.88]
      (0,1,0.16) -- (-0.433,0.25,0.16) -- (0.433,0.25,0.16) -- cycle;
  \draw[orange!70!black, thick]
      (0,1,0.16) -- (-0.433,0.25,0.16) -- (0.433,0.25,0.16) -- cycle;

  \fill[orange!24, fill opacity=0.88]
      (-0.433,0.25,0.16) -- (-0.866,-0.5,0.16) -- (0,-0.5,0.16) -- cycle;
  \draw[orange!70!black, thick]
      (-0.433,0.25,0.16) -- (-0.866,-0.5,0.16) -- (0,-0.5,0.16) -- cycle;

  \fill[orange!24, fill opacity=0.88]
      (0.433,0.25,0.16) -- (0,-0.5,0.16) -- (0.866,-0.5,0.16) -- cycle;
  \draw[orange!70!black, thick]
      (0.433,0.25,0.16) -- (0,-0.5,0.16) -- (0.866,-0.5,0.16) -- cycle;

  \draw[gray!50, dashed, thick]
      (-0.433,0.25,0.16) -- (0.433,0.25,0.16) -- (0,-0.5,0.16) -- cycle;

  \node[subtle point] at (0,1,0.32) {};
  \node[subtle point] at (-0.866,-0.5,0.32) {};
  \node[subtle point] at (0.866,-0.5,0.32) {};
  \node[subtle point] at (-0.433,0.25,0.32) {};
  \node[subtle point] at (0,-0.5,0.32) {};
  \node[subtle point] at (0.433,0.25,0.32) {};

  \draw[gray!45, thin] (0.866,-0.5,0)    -- (1.20,-0.12,0);
  \draw[gray!45, thin] (0.866,-0.5,0.16) -- (1.20,-0.12,0.16);
  \draw[eps small, <->]
      (1.20,-0.12,0) -- (1.20,-0.12,0.16)
      node[midway, right=2pt, eps label] {$\varepsilon/2$};

  \node[font=\large] at (0,0,-0.57) {$A(2)=\dfrac{A+A}{2}$};

\end{scope}

\end{tikzpicture}

    \caption{Construction of the set $A$ in dimension $n=3$ and the corresponding second Minkowski average $A(2)=\frac{A+A}{2}$.}

    \label{n=3pic}

\end{figure}
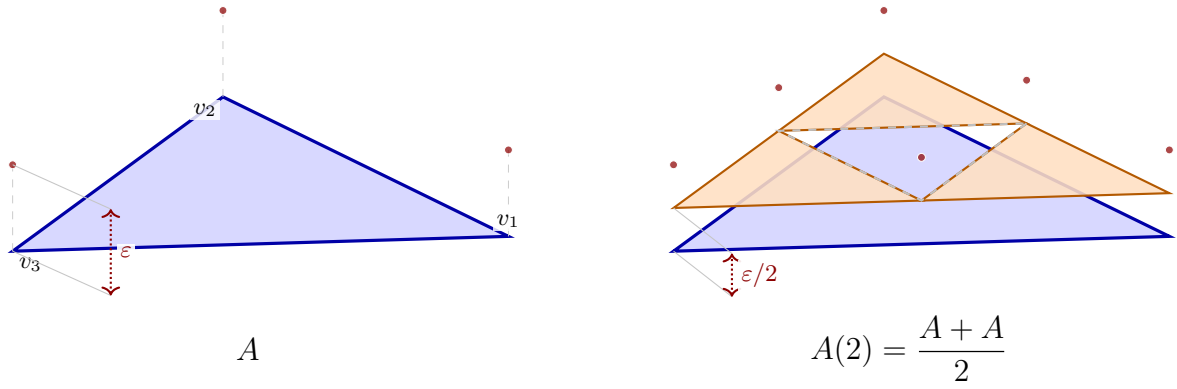

\begin{ack*}
The author is grateful to ChatGPT 5.5 Pro for helping him prove \Cref{ChatProp} and generating the TikZ code for the image of the construction for $n=3$. The responsibility for the content of this note lies exclusively with the author.
\end{ack*}

\section{Proofs}
\begin{proof}[Proof of \Cref{exampleprop}]
Let $T\subset\mathbb{R}^{n-1}\times\{0\}$ be a centered regular $(n-1)$-dimensional simplex with vertices $v_1,\dots,v_n$, so that $T$ is the convex hull of the $v_i$'s, i.e. $T=\left\{\sum \lambda_i v_i: \lambda_i\geq0, \sum \lambda_i=1\right\}$. Assume that $\|v_i\|_2=1$ for all $i$. Let $0<\epsilon<\frac{1}{n^3}$. Define $A\subset\mathbb{R}^n$ as follows:
$$A:=T\cup \{v_i+\epsilon e_n: i=1,\dots n\},$$
where $e_n$ denotes the last standard basis vector.
Note that $\co(A)=T+[0,\epsilon]e_n$, so that $d(A)=\epsilon$.

What remains is a simple computation to verify that the distance from $A(k)$ to the point $\epsilon e_n\in \co(A)=\co(A(k))$ is indeed $\epsilon$ for all $1\leq k\leq n-1$.

Consider points $a_1,\dots,a_k\in A$ and let $m:=\frac{a_1+\dots+a_k}{k}$. If $a_i\in T$ for all $i$, then $m\in T$, so clearly $\|m-\epsilon e_n\|_2\geq \epsilon$. Consider the case that at least one of the $a_i$'s is of the form $v_i+\epsilon e_n$, say $a_1=v_1+\epsilon e_n$. Write all the others as $a_i=\eta_i\epsilon e_n+\sum_j \lambda_{i,j}v_j$, with $\eta_i\in \{0,1\}$, $\lambda_{i,j}\geq 0$ and $\sum_j\lambda_{i,j}=1$. Now we can write 
$$km=\sum_{i=1}^ka_i=v_1+\epsilon e_n+\left(\sum_{i=2}^k \eta_i\right)\epsilon e_n +\sum_j\left(\sum_{i=2}^k\lambda_{i,j}\right)v_j.$$

First note that we may ignore the component of $m-\epsilon e_n$ in the $e_n$ direction, to find that
$$\|m-\epsilon e_n\|_2\geq \frac1k\left\|v_1+\sum_j\left(\sum_{i=2}^k\lambda_{i,j}\right)v_j\right\|_2 $$

Consider the smallest coordinate $L$ for some $v_j$ apart from $v_1$, and note that since $k\leq n-1$ we have
$$L:=\min_{2\leq j\leq n}\left(\sum_{i=2}^k\lambda_{i,j}\right)\leq \frac{\sum_{j=2}^{n}\sum_{i=2}^k\lambda_{i,j}}{n-1} \leq \frac{k-1}{n-1}\leq \frac{n-2}{n-1}.$$
Without loss of generality, say $L$ corresponds to the $v_n$ coefficient, i.e. $L=\sum_{i=2}^k\lambda_{i,n}$.
Recall that since $T$ is centered and regular, we have that $\sum_{j=1}^nv_j=0$, so that
$$v_1+\sum_{j=1}^n\left(\sum_{i=2}^k\lambda_{i,j}\right)v_j=v_1+\sum_{j=1}^n\left(\left(\sum_{i=2}^k\lambda_{i,j}\right)-L\right)v_j=\sum_{j=1}^{n-1}\mu_jv_j,$$
for some coefficients $\mu_j\geq0$. In particular, we have that $\mu_1\geq 1-L\geq \frac{1}{n-1}$. Hence, using that $\langle v_i,v_n\rangle=-\frac{1}{n-1}$ for all $i=1,\dots, n-1$, we find that 
$$\frac1k\left\|v_1+\sum_j\left(\sum_{i=2}^k\lambda_{i,j}\right)v_j\right\|_2= \frac1k\left\|\sum_{j=1}^{n-1}\mu_jv_j\right\|_2\geq 
\frac1k\left|\left\langle v_n,\sum_{j=1}^{n-1}\mu_jv_j\right\rangle\right|=\frac{1}{k(n-1)} \sum_{j=1}^{n-1} \mu_j\geq \frac{1}{k(n-1)^2}>\epsilon.$$
Hence, the origin is indeed the point in $A(k)$ from $\epsilon e_n\in \co(A)$, so that indeed $\operatorname{dist}(\epsilon e_n, A(k))=\epsilon$ for all $k\leq n-1$.
\end{proof}

We now turn to the proof of \Cref{ChatProp} which was found with the aid of ChatGPT 5.5 Pro. The argument is inspired by the eventual monotonicity theorem of Fradelizi, Madiman, Marsiglietti, and Zvavitch \cite{fradelizi2016minkowski} using \cite{schneider1975measure}, with a tweak to handle the points in the interior of $\co(A)$.

\begin{proof}[Proof of \Cref{ChatProp}]
First note that if $d(A)=0$, then $d(A(n))\leq d(A)=0$, so there is nothing to prove. Henceforth, assume $d(A)>0$.
 We aim to upper bound
\(\operatorname{dist}(x,A(n))\) uniformly for \(x\in \co(A)=\co(A(n))\). First note that if $x\in A(n)$, then \(\operatorname{dist}(x,A(n))=0\) so that the conclusion is immediate. Henceforth assume $x\not\in A(n)$.
We distinguish two cases according to whether $x\in \partial\co(A)$, or $x\in \operatorname{int}\co(A)$.

\begin{case}\label{case1}
    $x\in \partial\co(A)$
\end{case} 
Let $F$ be a face of $\co(A)$ containing \(x\) with \(m:=\dim F<n\). Note that we in particular have $x\in \co(A\cap F)$.  By Carath\'eodory's theorem, we can write 
$$x=\sum_{i=0}^m \lambda_i a_i,$$
with $a_i\in A\cap F$, $\lambda_i\geq 0$, and $\sum_{i=0}^m\lambda_i=1$.
Additionally consider the points
$$p:=\frac{\sum_{i=0}^m \lfloor n\lambda_i\rfloor a_i}{\sum_{i=0}^m \lfloor n\lambda_i\rfloor }\in A\left(\sum_{i=0}^m \lfloor n\lambda_i\rfloor \right)\cap F \qquad \text{ and }\qquad q:= \frac{\sum_{i=0}^m \left(n\lambda_i-\lfloor n\lambda_i\rfloor\right) a_i}{n-\sum_{i=0}^m \lfloor n\lambda_i\rfloor }\in F.$$
We first check these are well-defined. If $\sum_{i=0}^m \lfloor n\lambda_i\rfloor =n$, then $x=p\in A(n)$, which contradicts that $x\not\in A(n)$, so $n-\sum_{i=0}^m \lfloor n\lambda_i\rfloor\geq 1$ and $q$ is well-defined. Since $m\leq n-1$ and $\sum_{i=0}^m \lambda_i=1$, at least one of the $\lambda_i\geq 1/n$, so $\sum_{i=0}^m \lfloor n\lambda_i\rfloor\geq 1$ and $p$ is well-defined.

Note that $x=\frac{\left(\sum_{i=0}^m \lfloor n\lambda_i\rfloor\right) p + \left(n-\sum_{i=0}^m \lfloor n\lambda_i\rfloor\right) q}{n}$, but unfortunately $q$ may not be in $A(n-\sum_{i=0}^m \lfloor n\lambda_i\rfloor)$. To remedy this, note that $q\in \co(A)$, so there exists some $a\in A$ with $\|q-a\|_2\leq d(A)$. Consider the point
$$s:= \frac{\left(\sum_{i=0}^m \lfloor n\lambda_i\rfloor\right) p + \left(n-\sum_{i=0}^m \lfloor n\lambda_i\rfloor\right) a}{n}\in A(n).$$
Finally, we have
$$\|x-s\|_2=\left\|\frac{\left(n-\sum_{i=0}^m \lfloor n\lambda_i\rfloor\right) }{n} (q-a)\right\|_2\leq \frac{n-1}{n}\|q-a\|_2\leq \frac{n-1}{n}d(A),$$
which concludes this case.

\begin{case}
   $x\in \operatorname{int}\co(A)$ 
\end{case}
 Choose \(a\in A\) with
\(\|x-a\|_2\le d(A)\). If \(a=x\), then \(x\in A\subset A(n)\) and there is nothing to prove. Otherwise, let the ray from \(a\) through \(x\) meet \(\partial\co(A)\) at \(b\), and let $\lambda\in[0,1)$ so that
\(
x=\lambda a+(1-\lambda)b.
\)
Then
\[
\|a-b\|_2=\frac{\|a-x\|_2}{1-\lambda}\le \frac{d(A)}{1-\lambda}.
\]

If \(\lambda\ge 1/n\), then $z:=
\frac{x-\frac1n a}{1-\frac1n}\in [a,b] \subset \co(A)$.
Find \(a'\in A\) within distance \(d(A)\) of $z$ and let $p:=\frac{n-1}{n}a'+\frac{1}{n}a\in A(n)$. We find that

$$\|x-p\|_2=\left\| \left(\frac{n-1}{n}z+\frac{1}{n}a\right)-\left(\frac{n-1}{n}a'+\frac{1}{n}a\right)\right\|_2=\frac{n-1}{n}\|z-a'\|_2\leq \frac{n-1}{n}d(A).$$

It remains to consider \(0\le \lambda<1/n\). Find \(a''\in A\) with \(\|b-a''\|_2\le d(A)\) and let
$
q:=\frac1n a+\frac{n-1}{n}a''\in A(n).
$
We can compute
\[
\left\|
x-q
\right\|_2\leq \left\|
\left(\lambda a+(1-\lambda)b\right)-\left(\frac1n a+\frac{n-1}{n}a''\right)
\right\|_2
\le \left(\frac1n-\lambda\right)\|a-b\|_2+\frac{n-1}{n}\|b-a''\|_2
\le
\left(\frac{\frac1n-\lambda}{1-\lambda}
+
\frac{n-1}{n}\right)d(A).
\]
On the other hand, we know from \Cref{case1}, that since \(b\in\partial\co(A)\), there exists \(y\in A(n)\) with
$\|b-y\|_2\le \frac{n-1}{n}d(A),$
so
\[
\|x-y\|_2\leq \|x-b\|_2+\|b-y\|_2
\le \frac{\lambda}{1-\lambda} \|x-a\|_2+\frac{n-1}{n}d(A)
\le
\left(\frac{\lambda}{1-\lambda}
+
\frac{n-1}{n}\right)d(A).
\]
Combining these two bounds, we find
\[
\operatorname{dist}(x,A(n))\le \min\{\left\|
x-q
\right\|_2,\|x-y\|_2\}
\le
\left(
\frac{n-1}{n}
+
\min\left\{
\frac{\lambda}{1-\lambda},
\frac{\frac1n-\lambda}{1-\lambda}
\right\}
\right)d(A).
\]
For \(0\le \lambda<1/n\), it is easy to check that $
\min\left\{
\frac{\lambda}{1-\lambda},
\frac{\frac1n-\lambda}{1-\lambda}
\right\}
\le \frac1{2n-1}$,
with equality at \(\lambda=1/(2n)\). Thus
\[
\operatorname{dist}(x,A(n))
\le
\left(\frac{n-1}{n}+\frac1{2n-1}\right)d(A)=\left(1-\frac{n-1}{n(2n-1)}\right)d(A),
\]
which concludes the second case and thus the proof of \Cref{ChatProp}.
\end{proof}

\bibliographystyle{alpha}
\bibliography{references}

\end{document}